\documentclass[12pt,a4paper]{amsart}
\usepackage{graphicx}
\theoremstyle{plain}
\usepackage{amsfonts,amssymb,mathrsfs,amscd}
\usepackage[english]{babel}

\advance\hoffset-15mm \advance\textwidth30mm
\advance\voffset-12mm \advance\textheight24mm

\newtheorem{theorem}{Theorem}
\newtheorem{lemma}{Lemma}
\newtheorem*{theo*}{Theorem}
\theoremstyle{definition}
\newtheorem{definition}{Definition}
\newtheorem*{definition*}{Definition}

\newtheorem{remark}{Remark}

\newcounter{caseinproof}
0

\def\NN{{\mathbb N}}

\def\kk{{\Bbbk}}
\def\CC{{\mathbb C}}
\def\CM{\mathcal C}

\def\SAut{\mathop{\rm SAut}}

\def\reg{\mathop{\rm reg}}

\def\rk{\mathop{\rm rk}}

\def\Diag{\mathop{\rm Diag}}

\def\ll1{l_{\lambda}^{-1}(1)}
\def\lm1{l_{\mu}^{-1}(1)}

\begin{document}

\title{Infinite transitivity for Calogero-Moser spaces}
\author%
{Karine Kuyumzhiyan}
\address{National Research University Higher School of Economics, Usacheva~6, Moscow, Russia} 
\email{karina@mccme.ru}
\thanks{The study has been funded by the Russian Academic Excellence Project ``5-100''}

%
\begin{abstract} 
We prove the conjecture of Berest-Eshmatov-Eshmatov by showing that the group of automorphisms of a product of  Calogero-Moser spaces~$\CM_{n_i}$, where the $n_i$ are pairwise distinct, acts $m$-transitively for each~$m$.
\end{abstract}
\maketitle

\section{Introduction}
For affine algebraic varieties, their automorphism groups are usually small. However, if they are rich, such varieties and their automorphims groups become objects of intensive study. If an automorphism group is infinite dimensional, it may satisfy the property called {\it infinite transitivity}: for any $m\in \NN$ the group can map any $m$-tuple of points of the variety to any other $m$-tuple of points. We study Calogero-Moser spaces and their products and show that their automorphism groups are infinitely transitive.
\begin{definition}
The Calogero-Moser space $\CM_n$ is 
$$
\CM_n:=\{(X,Y)\,\in\, Mat_n(\CC)\times Mat_n(\CC)\; \colon \; \rk([X,Y]+I_n)=1\} /\!\!/ PGL_n(\CC),
$$
where $PGL_n(\CC)$ acts via $g. (X,Y)=(gXg^{-1}, gYg^{-1})$. 
\end{definition}

Calogero-Moser spaces play an important role in Representation Theory. It is known that $\CM_n$ is a smooth irreducible affine algebraic variety of dimension~$2n$, see Wilson~\cite{Wil}. It is rational, see~\cite[Prop. 1.10]{Wil} and~\cite[Remark 5]{Pop1}. It carries a symplectic structure, see~\cite{Et}. It is a particular case of a Nakajima quiver variety. It appears as a partial compactification of the Calogero-Moser integrable system.

\begin{definition}
We denote by~$G$ the group generated by two kinds of transformations.
\begin{equation}
(X,Y)\mapsto (X+p(Y), Y), \;p \mbox{ is a polynomial in one variable}, \label{triang_Y}
\end{equation}
\begin{equation}
(X,Y)\mapsto (X, Y+q(X)),\; q \mbox{ is a polynomial in one variable}. \label{triang_X}
\end{equation}
It is isomorphic to the group of automorphisms of the first Weyl algebra~\cite{Dixmier, ML}. 
\end{definition}

Formulae~\eqref{triang_Y} and~\eqref{triang_X} can be used to define the $G$-action on $Mat_n(\CC)\times Mat_n(\CC)$. This action descends to~$\CM_n$. To verify this, check two things. First, formulae~\eqref{triang_Y} and~\eqref{triang_X} agree with the $PGL_n(\CC)$-action. Second, the obtained points remain inside~$\CM_n$. Indeed, $[X+p(Y),Y]=[X,Y]=[X,Y+q(X)]$, hence,
$$
\rk([X+p(Y),Y]+I_n)=\rk([X,Y+q(X)]+I_n)=\rk([X,Y]+I_n)=1.
$$

\begin{theorem}(~\cite[Theorem~1]{BEE})
For each $n\geqslant 1$, the action of~$G$ on~$\CM_n$ is doubly transitive. 
\end{theorem}

The conjecture in~\cite{BEE} says, in particular, that~$\CM_n$ has an infinitely transitive action of its automorphism group. It is proved below in Theorem~\ref{MainTheorem}$a)$.

There is a more general class of varieties: for any pairwise distinct integers $n_1$, $n_2, \ldots$, $n_k$ one can consider the product of the corresponding Calogero-Moser spaces 
\begin{equation}\label{product_of_Cn}
\CM_{n_1}\times \CM_{n_2}\times \ldots \times \CM_{n_k}.
\end{equation}
 The group~$G$ acts diagonally on this product. It also acts on
 \begin{equation}\label{disjunion_of_Cn}
 \CM_{n_1}\sqcup \CM_{n_2}\sqcup \ldots \sqcup \CM_{n_k}.
 \end{equation}
  
 Moving a finite number of points on the product~\eqref{product_of_Cn} can be seen as moving a finite number of points on $\CM_{n_1}\sqcup \CM_{n_2}\sqcup \ldots \sqcup \CM_{n_k}$.
For these actions, we consider the property of {\it collective infinite transitivity}. 
\begin{definition}\label{def_of_collective}
We say that the $G$-action on~\eqref{product_of_Cn} or on~\eqref{disjunion_of_Cn} is {\it collectively infinite transitive} if for any integers $m_1$, $m_2, \ldots$, $m_k$ and for any two tuples of $m_1$ points on the first variety $\CM_{n_1}$, $m_2$ points on the second variety $\CM_{n_2}$, etc., $m_k$ points on the $k$th variety $\CM_{n_k}$ there exists an element of~$G$ which simultaneously maps the first tuple to the second tuple.
\end{definition}

\begin{theorem}~(\cite[Theorem~2]{BEE})
For any pairwise distinct natural numbers $(n_1,n_2, \ldots,n_k)\in\NN^k$, the diagonal action of~$G$ on $\CM_{n_1}\times \CM_{n_2}\times \ldots \times \CM_{n_k}$ is transitive.
\end{theorem}
If $n_i=n_j$, then $\CM_{n_i}\times \CM_{n_j}$ has a diagonal subvariety which remains invariant under the diagonal $G$-action.

The Conjecture in~\cite{BEE} states that the $G$-action on $\CM_{n_1}\times \CM_{n_2}\times \ldots \times \CM_{n_k}$  is collectively  infinitely transitive. We prove this conjecture in Theorem~\ref{MainTheorem}$b)$.

The key ingredient of the proof is that, whenever $X$-components of the given points have pairwise coprime minimal polynomials, the given points can be moved independently via automorphisms of the form~\eqref{triang_X}.


The author is grateful to Michail Zaidenberg for pointing out this problem and for many useful remarks, to Ivan Arzhantsev for valuable comments, to Yuri Berest for pointing out some inaccuracy in the first version of the text, to Michail Bershtein, Alexander Perepechko, and Misha Verbitsky for encouraging. The author thanks Institut Fourier, Grenoble, for its hospitality.

\section{Geometry of Calogero-Moser spaces and their automorphisms}
We recall here some facts on the geometry of~$\CM_n$ from~\cite[Sec. 1]{Wil} and then  strengthen them to apply to the products of Calogero-Moser spaces of the forms~\eqref{product_of_Cn} and~\eqref{disjunion_of_Cn}. We also use results of Berest and Wilson~\cite{BW}.


\begin{lemma}\label{Form}(\cite[Prop. 1.10]{Wil})
If $(X,Y)\in \CM_n$ and if $X$ is diagonal, then the eigenvalues $x_1,x_2,\ldots,x_n$ of~$X$ are distinct. The non-diagonal entries of
$Y$ have the form 
$$
y_{ij}=1/(x_i-x_j).
$$
\end{lemma}
If $X$ is diagonalizable but not diagonal, 
the point $(X,Y)$ of $\CM_n$ has another representative $(AXA^{-1}, AYA^{-1})$ where the new $X$ is diagonal and we can express all the non-diagonal entries of~$Y$ in entries of~$X$.


\begin{lemma}(\cite[Lemma 10.2]{BW})\label{interpolation}
If $(X,Y)\in \CM_n$ and $(X,Y')\in \CM_n$ with $X$ diagonal, then there exist a polynomial~$p$ in one variable such that $(X,Y)\mapsto (X,Y+p(X))=(X,Y')$.
\end{lemma}
\begin{proof}
By Lemma~\ref{Form}, matrices $Y$ and $Y'$ may differ only in diagonal entries, denote them by~$y_{11},\ldots,y_{nn}$ and~$y'_{11},\ldots,y'_{nn}$. Let $X=\Diag (x_1,\ldots,x_n)$. Since all the $x_i$ are different, there exists an interpolation polynomial  $p(x)$ such that $p(x_i)=y'_{ii}-y_{ii}$. But $p(X)=\Diag (p(x_1),\ldots, p(x_n))$ and hence $Y+p(X)=Y'$.
\end{proof}

%

\begin{remark}\label{modulo_char}
Let $(X_0,Y_0)\in\CM_n$. If a polynomial in one variable $q(x)$ is divisible by the minimal polynomial of~$Y_0$, 
then the automorphisms $(X,Y)\mapsto (X+p(Y), Y)$ and $(X,Y)\mapsto (X+p(Y)+q(Y), Y)$ map $(X_0,Y_0)$ to the same point.
\end{remark}

\begin{lemma}\label{CollectiveAutomorphism}
Suppose that square matrices $X_1$, $X_2$, $\ldots$, $X_m$ (possibly of different sizes) have pairwise coprime minimal polynomials. Take $(X_1, Y_1)$,  $(X_2, Y_2)$, $\ldots$, $(X_m,Y_m)$, where each~$Y_i$ is a square matrix of the same size as~$X_i$,  and polynomials $p_1,p_2, \ldots, p_m \in \kk[x]$.
Then there exists a polynomial $p\in\kk[x]$ such that for each~$i$ we have
$$
Y_i+p_i(X_i)=Y_i+p(X_i).
$$ 
\end{lemma}

\begin{proof}
By Remark~\ref{modulo_char}, each $p_i$ is defined modulo the minimal polynomial $\chi_i$ of $X_i$. Since $\chi_1$, $\chi_2,\ldots$,$\chi_m$ are pairwise coprime, by the Chinese remainders theorem there exists a polynomial $p$ such that for each~$i=1,2,\ldots,m$ the polynomial $p-p_i$ is divisible by~$\chi_i$. 
\end{proof}

\begin{lemma}[Refinement of Lemma~\ref{interpolation}]\label{RefineInterp}
Take two $m$-tuples of points of $\CM_{n_1}\sqcup \CM_{n_2}\sqcup \ldots \sqcup \CM_{n_k}$
$$
(X_1,Y_1), (X_2,Y_2),\ldots,(X_m,Y_m)
$$
and
$$
(X_1,Y'_1), (X_2,Y'_2),\ldots,(X_m,Y'_m)
$$
(so, each $\CM_{n_i}$ contains an even number of chosen points). Suppose that $X_1,\ldots,X_m$ are diagonalizable and have pairwise coprime minimal polynomials (equivalently, diagonalizable and with disjoint spectra). Then there exists a polynomial $p(x)\in\kk[x]$ such that for each~$i$ we have
$$
Y_i+p_i(X_i)=Y'_i.
$$ 
\end{lemma}
\begin{proof}
First, by Lemma~\ref{interpolation} we choose a polynomial $p_i(x)\in\kk[x]$ such that $Y'_i=Y_i+p(X_i)$ for each $i=1,2,\ldots,m$. Then by Lemma~\ref{CollectiveAutomorphism} we find a polynomial $p(x)$ which works for all~$i$.
\end{proof}

The following lemma is a refined Lemma 10.3 from~\cite{BW}. Its proof is explained in \cite[Lemma 5.6]{Wil} and also in \cite{Shi}, \cite[Prop. 8.6]{SW}.

\begin{lemma}\label{exists_poly}
Let $(X,Y)\in \CM_n$. Then there exists a polynomial $p$ such that the matrix $X+ p(Y)$ is diagonalizable.
\end{lemma}

By {\it almost all} we mean a cofinite subset of the set of complex numbers, i.e., all complex numbers but finitely many. 
We prove the following generalization of Lemma~\ref{exists_poly}.


\begin{lemma}\label{EverythingDiagonalizable}
a) Let $(X,Y)\in \CM_n$. Then there exists a polynomial $p$ such that the matrix $X+t\cdot p(Y)$ is diagonalizable for almost all~$t$.

b) Let us fix~$m$, $m\in\NN$, and take an $m$-tuple of points of $\CM_n$. Then one can make all the $2m$ matrices diagonalizable via a composition of $2m$ automorphisms of the forms~\eqref{triang_Y} and~\eqref{triang_X}. 

c) Let us take $m_1$ points on the first variety $\CM_{n_1}$, $m_2$ points on the second variety $\CM_{n_2}$, etc., $m_k$ points on the $k$th variety $\CM_{n_k}$ . Then all the matrices (i.e., $X$- and $Y$-components of our points) can be made diagonalizable via a composition of $2(m_1+m_2+\ldots+m_k)$ automorphisms of the forms~\eqref{triang_Y} and~\eqref{triang_X}.
\end{lemma}

\begin{remark}\label{Resultant}
Given a matrix, 
the condition of it having simple spectrum can be expressed as the condition of non-vanishing of some polynomial in the matrix entries.
Indeed, we compute the resultant of the characteristic polynomial of the matrix and of its derivative. 
If the resultant is nonzero, then these polynomials have no common roots, hence, the characteristic polynomial of the matrix cannot have multiple roots.  

The condition of having a simple spectrum for $X$ (or, equivalently, for $Y$) implies diagonalizability. Suppose that the matrix $X+t\cdot p(Y)$ is diagonalizable  for some $t=t_0$. Then the above resultant is a nonzero polynomial in~$t$, hence,  its values at almost all $t$ are nonzero and the matrix $X+t\cdot p(Y)$ is diagonalizable for almost all~$t$.
\end{remark}

\begin{proof}
a) Take a polynomial~$p$ as in Lemma~\ref{exists_poly}. By Remark~\ref{Resultant}, the matrix $X+t\cdot p(Y)$ is diagonalizable for almost all~$t$ since it is so for~$t=1$.

b) Using Lemma~\ref{exists_poly}, make $X_1$ diagonalizable.  Then, acting as in a), find a polynomial $p_2$ such that $X_2+p_2(Y_2)$ is diagonalizable. Consider automorphisms $(X,Y)\mapsto (X+t\cdot p_2(Y))$, $t\in \CC$. By Remark~\ref{Resultant}, $X_1$ maps to a diagonalizable matrix via such automorphisms for almost all~$t$  since it is so for~$t=0$. Also, by Remark~\ref{Resultant}, the matrix $X_2$ maps to a diagonalizable matrix for almost all~$t$ since it is so for~$t=1$. For the values of~$t$, we forbid a union of two finite sets hence a finite set. Choose any other~$t$, the images of $X_1$ and $X_2$ are diagonalizable for it. In this way we make all the $X_i$ diagonalizable one by one. Then in the same way we make all the $Y_i$ diagonalizable, while the $X_i$ remain unchanged and hence diagonalizable.

c) The proof is exactly the same as in b).
\end{proof}

There is a map $\Upsilon \colon \CM_n \rightarrow (\CC^n/S_n)\times (\CC^n/S_n)$ which sends $X$ and $Y$ to their spectra, where $S_n$ stands for the symmetric group on an $n$-element set. By $\Upsilon_1$ and $\Upsilon_2$ we mean projections to the first and to the second components, respectively.
One of the key statements is

\begin{lemma}[Prop. 4.15 and Theorem 11.16 in~\cite{EG}]
The map $\Upsilon$ is surjective. 
\end{lemma}

\begin{lemma}\label{Surjectivity}
Take an $n\times n$ matrix $Y$ with a simple spectrum $(\mu_1,\mu_2,\ldots, \mu_n)$. Fix pairwise distinct $\lambda_1,\lambda_2,\ldots,\lambda_n\in\CC$. Then there exists $(X,Y)\in\CM_n$ such that $X$ has eigenvalues $\lambda_1,\lambda_2,\ldots,\lambda_n$.
\end{lemma}
\begin{proof}
Since $\Upsilon$ is surjective, there is a point $(X',Y')$ such that 
$$
\Upsilon((X',Y'))=((\lambda_1,\lambda_2,\ldots,\lambda_n),(\mu_1,\mu_2,\ldots, \mu_n)).
$$
 Since $\mu_i$ are pairwise distinct, $Y'$ is conjugate to~$Y$, that is, there exists a matrix~$A$ such that $Y=AY'A^{-1}$. Take $X=AX'A^{-1}$. Clearly, $(X,Y)$ is the same point of $\CM_n$ as $(X',Y')$ and $X$ has the prescribed spectrum.
\end{proof}

\begin{remark}
In~\cite{BEE}, the fibers of $\Upsilon_1$ over nilpotent Jordan blocks are used. The advantage is that $X^n=0$. We use the fibers over diagonalizable~$X$ (hence having simple spectra) since they can be easily described.
\end{remark}

\section{Main results}

We are ready to prove our main result.
\begin{theorem}\label{MainTheorem}
a) The group of automorphisms of a Calogero-Moser space~$\CM_n$ acts infinitely transitively.

b) The group of automorphisms of a product of Calogero-Moser spaces $\CM_{n_1}\times \CM_{n_2}\times \ldots \times \CM_{n_k}$, where $n_1,n_2,\ldots,n_k$ are pairwise distinct, acts collectively infinitely transitively. 
\end{theorem}

We prove these statements together since their proofs are almost identical.

We use the two-transitivity of the $G$-action on~$\CM_n$ and on $\CM_{n_1}\times \CM_{n_2}\times \ldots \times \CM_{n_k}$ which is established in \cite{BEE}. 
On $\CM_{n_1}\times \CM_{n_2}\times \ldots \times \CM_{n_k}$, the two-transitivity can mean two different things. First, when two points are in the same $\CM_{n_i}$, then the two-transitivity on the product follows from the two-transitivity on $\CM_{n_i}$ proved in~\cite[Theorem 1]{BEE}. Second, when two points belong to different $\CM_{n_i}$ and $\CM_{n_j}$, then it follows from~\cite[Theorem 2]{BEE}.

\begin{proof}

Step 1. Suppose that we want to map one $m$-tuple of points $(X_1, Y_1)$,  $(X_2, Y_2), \ldots$, $(X_m,Y_m)$ to another $m$-tuple of points $(X_{m+1}, Y_{m+1})$,  $(X_{m+2}, Y_{m+2}), \ldots$, $(X_{2m},Y_{2m})$. By Lemma~\ref{EverythingDiagonalizable}, there exists an automorphism making all the $4m$ matrices diagonalizable.

Step 2. 
Let us show that the spectra of all the $X_i$ can be assumed to be disjoint and, simultaneously, those of the $Y_i$ can be also assumed to be disjoint via several extra automorphisms. For this, we draw a graph on~$2m$ vertices. An edge $ij$ is drawn if and only if 

\begin{center}
\emph{($X_i$ and $X_j$ have no common eigenvalue)}\&\emph{($Y_i$ and $Y_j$ have no common eigenvalue)}. 
\end{center}

First we prove Theorem~\ref{MainTheorem}a). Let us construct the first edge. We fix two pairs $(X_1^0, Y_1^0)$ and $(X_2^0, Y_2^0)\in \CM_n$ with disjoint spectra and by two-transitivity find a sequence of polynomials such that the corresponding composition of automorphisms of the forms~\eqref{triang_Y} and~\eqref{triang_X}
 maps $(X_1,Y_1)$ and $(X_2,Y_2)$ there. 
We can regard each particular automorphism as an element of its one-parameter subgroup with $t=1$. Varying $t$ as in Remark~\ref{Resultant}, we see that for almost all~$t$ the images of $(X_1,Y_1)$ and $(X_2,Y_2)$ will have no common eigenvalue, and for almost all~$t$ all the matrices will remain diagonalizable. To obtain the first edge, we take any~$t$ satisfying all these conditions (we forbid a finite number of finite sets). 

Now let us create new edges. If $i$ and $j$ are not joined because $X_i$ and $X_j$ have a common eigenvalue, then find a simple spectrum for $X'_j$ disjoint from the spectra of all the other $X_k$. 
By Lemma~\ref{Surjectivity} there is a pair $(X'_j,Y_j)\in \CM_n$ with the prescribed spectrum for $X'_j$. Using 2-transitivity, find an automorphism mapping $(X_i,Y_i)$ to $(X_i,Y_i)$ and $(X_j,Y_j)$ to $(X'_j,Y_j)$. As above, we decompose it into automorphisms of the forms~\eqref{triang_Y} and~\eqref{triang_X} and regard it as an element of a one-parameter family of automorphisms with $t=1$ (not a subgroup!). 
We want all the matrices to remain diagonalizable, this forbids a finite number of values of~$t$. We do not want to break edges that were constructed earlier, so for each old edge $kl$, as we did in Remark~\ref{Resultant}, we express the condition
 that
 
 \emph{($X_k$ and $X_l$ have no common eigenvalue)}\&\emph{($Y_k$ and $Y_l$ have no common eigenvalue)}
 
\noindent as a polynomial condition on~$t$ that holds for $t=0$. We also forbid a finite number of $t$s checking that the spectrum of $X'_j$ is disjoint from the spectra of the images of all the other $X_k$, this was true for $t=0$. All in total, this is a finite number of restrictions on~$t$, and we can choose any other $t\in\CC$.  Then we perform  the same to disconnect spectra of $Y_i$ and $Y_j$. We obtain an edge between $i$ and $j$. We construct new edges in this way until we get a complete graph (i.e., any two vertices are joined by an edge). We further assume that all the spectra of $X_i$ are disjoint and all the spectra of $Y_j$ are disjoint.

Step 2 for Theorem~\ref{MainTheorem}b) is proved similarly. When we need 2-transi\-tivity for points in one component, we rely on~\cite[Theorem 1]{BEE}, and when we need it for two points from different components, we use~\cite[Theorem 2]{BEE}.
\medskip

Step 3. To obtain the $m$-transitivity, let us take two $m$-tuples of points $(X_1, Y_1)$,  $(X_2, Y_2), \ldots$, $(X_m,Y_m)$ and $(X_{m+1}, Y_{m+1})$,  $(X_{m+2}, Y_{m+2}), \ldots$, $(X_{2m},Y_{2m})$ on~$\CM_n$ and perform on this $2m$-tuple both Steps 1 and 2. We denote the new points by $(\tilde X_1, \tilde Y_1)$,  $(\tilde X_2, \tilde Y_2), \ldots$, $(\tilde X_m,\tilde Y_m)$ and $(\tilde X_{m+1}, \tilde Y_{m+1})$,  $(\tilde X_{m+2}, \tilde Y_{m+2}), \ldots$, $(\tilde X_{2m},\tilde Y_{2m})$.  We also denote by~$g$ the corresponding element of~$G$, i.e., $g.(X_i,Y_i)=(\tilde X_i, \tilde Y_i)$ for $i=1,2,\ldots,2m$.  
Let us choose representatives with all the $\tilde X_i$ diagonal.

Now we need the interpolation polynomial. We know how a triangular automorphism $(X,Y)\mapsto (X, Y+p(X))$ looks like: 
the non-diagonal elements of all the $\tilde Y_i$ do not change, and the $k$th diagonal element of the corresponding $\tilde Y_i$ increases by $p(\lambda_{ki})$, where $\lambda_{ki}$ is the $k$th diagonal element of the matrix $\tilde X_i$. 
\medskip

Using Lemma~\ref{Surjectivity}, find $m$ intermediate points of $\CM_n$ 
$$
(\tilde X_1, Y''_1), \mbox{ where } Y''_1 \mbox{ has the same spectrum as }\tilde Y_{m+1};
$$ 
$$
\ldots;
$$ 
$$
(\tilde X_m,Y''_m),\mbox{ where } Y''_m \mbox{ has the same spectrum as }\tilde Y_{2m}.
$$ 
By Lemma~\ref{RefineInterp}, there is an automorphism $Y\mapsto Y+p(X)$ which maps each $\tilde Y_i$, $1\leqslant i \leqslant m$, to the chosen matrix $Y''_i$. 

Now for each point choose a representative with $Y$ diagonal and make the same interpolation with $X$ and $Y$ reversed. 

Let us denote by~$g_1$ the corresponding element of~$G$, i.e., such that $g_1.(\tilde X_i, \tilde Y_i)=(\tilde X_{m+i},\tilde Y_{m+i})$ for $i=1,2,\ldots,m$. Then $g^{-1}g_1g$ maps $(X_1, Y_1)$,  $(X_2, Y_2), \ldots$, $(X_m,Y_m)$ to $(X_{m+1}, Y_{m+1})$,  $(X_{m+2}, Y_{m+2}), \ldots$, $(X_{2m},Y_{2m})$.
\end{proof}

\emph{Final remarks.}
For a variety~$X$, one can generate a group by all the one-parameter unipotent subgroups of $\mathrm{Aut}(X)$. This subgroup denoted by $\SAut(X)$ 
is treated in~\cite{KZ, AKZ, AFKKZ, AFKKZ2}. It is shown in~\cite{AFKKZ} that infinite transitivity of~$\SAut(X)$ on the smooth locus $\reg(X)$ for $\dim X\geqslant 2$ is equivalent to simple transitivity and is equivalent to \emph{flexibility property} which means that the tangent space $T_xX$ in every smooth point $x\in X$ is generated by tangent vectors to the orbits of one-parameter unipotent subgroups. We fix attention that this fact is not easily applicable to~$\CM_n$ since natural automorphisms $(X,Y)\mapsto (X+p(Y), Y)$ and $(X,Y)\mapsto (X, Y+q(X))$ do not come with all their (one-parameter unipotent) rescalings. 

On the other hand, it is not known whether the group~$G$ coincides with $\SAut(\CM_n)$.

\bigskip

\end{document}